\documentclass[preprint, 12pt]{elsarticle}
\usepackage{amssymb, amsthm, mathtools, mathrsfs, amsmath}
\usepackage[hidelinks]{hyperref}
\newtheorem{theorem}{Theorem}[section]
\newtheorem{lemma}[theorem]{Lemma}

\theoremstyle{remark}
\newtheorem{remark}[theorem]{Remark}
\numberwithin{equation}{section}

\usepackage{comment}

\DeclareMathOperator{\dist}{dist}

\begin{document}

\begin{frontmatter}

\title{Boundary blow-up solutions to real $(N-1)$-Monge-Amp\`{e}re equations with singular weights}

\author{Kiran Kumar Saha}
\ead{kiran_saha@iitg.ac.in}
\author{Sweta Tiwari\corref{cor1}}
\ead{swetatiwari@iitg.ac.in}
\cortext[cor1]{Corresponding author.}
\address{Department of Mathematics, Indian Institute of Technology Guwahati, Guwahati, \\ Assam 781039, India}

\begin{abstract}
In this paper, we study a boundary blow-up problem for real $(N-1)$-Monge-Amp\`{e}re equations of the form
\begin{equation}
	\nonumber  \left \{
	\begin{aligned}
		&  \operatorname{\det}^{\frac{1}{N-1}}\left(\Delta zI-D^{2}z\right)=K(|x|)f(z) && \text{ in } \Omega,\\
		&  z(x) \to \infty  \text{ as } \dist(x,\partial\Omega) \to 0,
	\end{aligned}
	\right.
\end{equation}
where $\Omega$ denotes a ball in $\mathbb{R}^{N} ~ (N \geq 2)$. The weight function $K$ is allowed to be singular, and the nonlinearity $f$ is assumed to satisfy a Keller-Osserman type condition. We establish the existence of infinitely many radial $(N-1)$-convex solutions to the system by employing the method of sub- and super-solutions, in conjunction with a comparison principle.
\end{abstract}


\end{frontmatter}

\section{Introduction}
\label{SectionIntroduction}

The Monge-Amp\`{e}re equation is a fully nonlinear second-order partial differential equation that arises in geometry, analysis, and applied mathematics, with key applications in the prescribed Gauss curvature problem, optimal transport, geometric optics, and large-scale fluid flows (\cite{TrudingerWang2008, Figalli2017, Gutierrez2001}). Indeed, for a function $z$ of class $C^{2}$, the real Monge-Amp\`{e}re operator $z \mapsto \det \left(D^{2}z\right)$, the determinant of the Hessian matrix $D^{2}z$, is itself a fully nonlinear second-order differential operator, as it depends nonlinearly upon the second-order derivatives. Moreover, it inherently determines the product of the eigenvalues of $D^{2}z$ and is elliptic when $D^{2}z$ is positive definite.

In this paper, we study the existence of infinitely many radial $(N-1)$-convex solutions to the real $(N-1)$-Monge-Amp\`{e}re equation
\begin{equation}\label{Eqne1}
	\left \{
	\begin{aligned}
		&  \operatorname{\det}^{\frac{1}{N-1}}\left(\Delta zI-D^{2}z\right)=K(|x|)f(z) && \text{ in } \Omega,\\
		&  z(x) \to \infty  \text{ as } \dist(x,\partial\Omega) \to 0,
	\end{aligned}
	\right.
\end{equation}
where $\Omega$ is a ball in $\mathbb{R}^{N} ~ (N \geq 2)$, $\operatorname{\det}^{\frac{1}{N-1}}(\cdot)$ is the $(N-1)^{\text{th}}$ root of $\det(\cdot)$, and $I$ is the $N \times N$ identity matrix. Without loss of generality, we assume in \eqref{Eqne1} that $\Omega$ is the unit ball in $\mathbb{R}^{N}$. Here, the weight function $K$, which may have a singularity at $1$, and the nonlinearity $f$ are assumed to satisfy: 
\begin{enumerate}
	\item[$(\mathbf{K})$] $K : [0,1) \to (0,\infty)$ is continuous on $[0,1)$;
	\item[$(\mathbf{f_{1}})$] $f : (0,\infty) \to (0,\infty)$ is nondecreasing and locally Lipschitz continuous on $(0,\infty)$;
	\item[$(\mathbf{f_{2}})$] $f$ satisfies the Keller-Osserman type condition
	\begin{align}\label{SecEqu01}
		\int^{\infty}\left(F(\tau)\right)^{-\frac{N-1}{2N-1}}d\tau=\infty,
	\end{align}
	where
	\begin{align}
		\nonumber  F(\tau) \coloneqq \int_{0}^{\tau}f(s)ds \quad \text{ for } \tau>0.
	\end{align}
\end{enumerate}

As usual, for a function $z \in C^{2}(\mathbb{R}^{N})$, we denote by $\Delta z=\sum_{i=1}^{N}\frac{\partial^{2}z}{\partial x_{i}^{2}}$ and $D^{2}z=\begin{pmatrix}
\frac{\partial^{2}z}{\partial x_{j} \partial x_{k}}
\end{pmatrix}$ for $j,k=1,2,\ldots,N$ the Laplacian and the Hessian matrix of $z$, respectively. Let $\lambda(D^{2}z)=(\lambda_{1},\lambda_{2},\ldots,\lambda_{N})$ be the vector of eigenvalues of $D^{2}z$, and define $\sigma_{i}=\sum\limits_{1 \leq j \leq N, j \neq i} \lambda_{j}$ for $i=1,2,\ldots,N$. Notice that $\sigma=(\sigma_{1},\sigma_{2},\ldots,\sigma_{N})$ is the vector of eigenvalues of $\Delta zI-D^{2}z$. The $(N-1)$-Monge-Amp\`{e}re operator is then given by $\det\left(\Delta zI-D^{2}z\right)=\prod_{i=1}^{N}\sigma_{i}$. For $N=2$, it follows that $\det\left(\Delta zI-D^{2}z\right)=\det\left(D^{2}z\right)$, which is the classical Monge-Amp\`{e}re operator in two dimensions. 
Following \cite{JiangJiLi2025}, a function $z \in C^{2}(\mathbb{R}^{N})$ is called $(N-1)$-convex if, for every $x \in \mathbb{R}^{N}$, the matrix $\Delta zI-D^{2}z>0$. We refer the reader to \cite{JiangJiLi2025, JiDengJiang2025} for further details.

A boundary blow-up problem is one in which a solution satisfies $z(x) \to \infty$ as $x \in \Omega$ and $\dist(x,\partial\Omega) \to 0$. Such a solution is termed a boundary blow-up (or large) solution, and this boundary behavior is often denoted, for conciseness, by $z=\infty$ on $\partial \Omega$. Boundary blow-up problems naturally arise in physics, chemistry, biology, and geometry, modeling phenomena such as singular heat or chemical concentrations, unbounded population densities, and large solutions of Monge-Amp\`{e}re or $k$-Hessian equations (\cite{Harada2015, BelowMailly2003, McCueElHachemSimpson2022}). The study of boundary blow-up problems has a long history, dating back to Bieberbach \cite{Bieberbach1916}, who investigated $\Delta z=e^{z}$ in a smooth bounded domain of $\mathbb{R}^{2}$. Chuaqui et al. \cite{ChuaquiCortazarElguetaFloresLetelierGarciaMelian2003} studied the boundary blow-up elliptic problem $\Delta z=K(x)z^{m}$ with $m>0$, establishing results on the existence, multiplicity, and boundary behavior of positive radial solutions. For further results on elliptic problems with boundary blow-up, see \cite{ChuaquiCortazarElguetaGarciaMelian2004} and the references therein. In $1957$, Keller \cite{Keller1957} and Osserman \cite{Osserman1957} proved independently that, for a positive, continuous, and nondecreasing function $f$, the equation $\Delta z=f(z)$ in $\mathbb{R}^{N}$ admits a large subsolution if and only if $f$ satisfies
\begin{align}
	\nonumber  \int^{\infty}\left(\int_{0}^{\tau}f(s)ds\right)^{-\frac{1}{2}}d\tau = \infty,
\end{align}
which is now known as the Keller-Osserman condition. Chen and Wang \cite{ChenWang2013} analyzed a boundary blow-up $p$-Laplacian problem with a weakly superlinear nonlinearity and improved the asymptotic description of large solutions.

There has been parallel development on boundary blow-up Monge-Amp\`{e}re problems in different settings. Zhang and Feng \cite{ZhangFeng2022} studied the boundary blow-up Monge-Amp\`{e}re problem with a gradient term
\begin{equation}\label{Eqne0001}
	\left \{
	\begin{aligned}
		&  \det\left(D^{2}z\right)=K(x)f(z)|\nabla z|^{q}, \quad q \geq 0, && \text{ in } \Omega,\\
		&  z=\infty && \text{ on } \partial \Omega,
	\end{aligned}
	\right.
\end{equation}
establishing sharp conditions on $K$, $f$, and $q$ for the existence and asymptotic behavior of strictly convex solutions, and deriving corresponding results for strictly convex radial solutions. When $q=0$ in \eqref{Eqne0001}, the problem reduces to the boundary blow-up Monge-Amp\`{e}re problem, which was studied by Zhang and Du \cite{ZhangDu2018} and by Lazer and McKenna \cite{LazerMcKenna1996} for $f(z)=e^{z}$ or $f(z)=z^{p}$ with $p>N$.

The $(N-1)$-Monge-Amp\`{e}re operator, which originated in complex geometry \cite{Gauduchon1984}, has seen a recent surge of interest. For Monge-Amp\`{e}re equations involving $(N-1)$-plurisubharmonic functions on compact K\"{a}hler manifolds, see \cite{TosattiWeinkove2016}. Jiang et al. \cite{JiangJiLi2025} studied the real $(N-1)$-Monge-Amp\`{e}re equation
\begin{align}
	\nonumber  \operatorname{\det}^{\frac{1}{N}}\left(\Delta zI-D^{2}z\right)=K(x)f(z) \quad \text{ in } \mathbb{R}^{N},
\end{align}
establishing necessary and sufficient conditions on $K$ and $f$ for the existence of entire subsolutions, as well as sufficient conditions for entire radial and bounded solutions. In \cite{JiDengJiang2025}, Ji et al. investigated the real $(N-1)$-Monge-Amp\`{e}re boundary blow-up problem
\begin{equation}
	\nonumber  \left \{
	\begin{aligned}
		& \operatorname{\det}^{\frac{1}{N}}\left(\Delta zI-D^{2}z\right)=f(x,z) && \text{ in } \Omega,\\
		&   z=\infty && \text{ on } \partial \Omega,
	\end{aligned}
	\right.
\end{equation}
establishing results on the existence, uniqueness, and asymptotic behavior of boundary blow-up solutions.

The Hessian matrix plays a central role in the formulation of second-order partial differential equations; in particular, its trace and determinant correspond to the Laplacian and Monge-Amp\`{e}re operators, respectively. For $k \in \{1,2,\ldots,N\}$, the $k$-Hessian operator of $z$ is the $k^{\text{th}}$ elementary symmetric function of the eigenvalues of $D^{2}z$,
\begin{align}
	\nonumber  S_{k}\left(\lambda(D^{2}z)\right) \coloneqq \sum_{1 \leq i_{1} < i_{2} < \cdots < i_{k} \leq N} \lambda_{i_{1}} \lambda_{i_{2}} \cdots \lambda_{i_{k}},
\end{align}
which reduces to the Laplacian when $k=1$ and to the Monge-Amp\`{e}re operator when $k=N$. There has been much recent interest in studying $k$-Hessian equations. Ji and Bao \cite{JiBao2010}, in analogy with the Keller-Osserman condition, established a necessary and sufficient condition under which the $k$-Hessian equation $S_{k}^{1/k}\left(\lambda(D^{2}z)\right)=f(z)$ in $\mathbb{R}^{N}$ admits a positive entire subsolution. Zhang and Feng \cite{ZhangFeng2023} studied the boundary blow-up $k$-Hessian problem with a gradient term
\begin{equation}\label{Eqne0002}
	\left \{
	\begin{aligned}
		&  S_{k}\left(\lambda(D^{2}z)\right)=K(x)f(z)|\nabla z|^{q}, \quad q \geq 0, && \text{ in } \Omega,\\
		&  z=\infty && \text{ on } \partial \Omega,
	\end{aligned}
	\right.
\end{equation}
establishing necessary and sufficient conditions for the existence of $k$-convex solutions, investigating their boundary asymptotic behavior, and presenting results for radially symmetric $k$-convex solutions. The special case $q=0$ of problem \eqref{Eqne0002} has also emerged as a cornerstone of contemporary research. In particular, Zhang and Feng \cite{ZhangFeng2019} and Ma and Li \cite{MaLi2019} investigated boundary blow-up problems for $k$-Hessian equations with singular weights. Zhang and Feng \cite{ZhangFeng2018JMAA} considered the weakly superlinear case $f(z)=z^{k}(\ln z)^{p}$ with $p>0$, and later, in \cite{FengZhang2020}, they refined the study by filling certain parameter gaps and allowing singular weights. In \cite{ZhangFeng2018NA}, the focus was on the positive power-type nonlinearity $f(z)=z^{p}$ with $p>0$. Recently, using the sub- and super-solutions method, Zhang et al. \cite{ZhangZhangWangAhmad2023} studied the existence of infinitely many radial $k$-convex solutions to
\begin{equation}
	\nonumber  \left \{
	\begin{aligned}
		&  S_{k}\left(\lambda(D^{2}z-\mu I)\right)=K(|x|)f(z), \quad \mu \geq 0, && \text{ in } \Omega,\\
		&  z=\infty && \text{ on } \partial \Omega,
	\end{aligned}
	\right.
\end{equation}
and Feng and Zhang \cite{FengZhang2023} investigated the existence of infinitely many radial $p$-$k$-convex solutions to
\begin{equation}
	\nonumber  \left \{
	\begin{aligned}
		&  S_{k}\left(\lambda\left(D_{i}(|Dz|^{p-2}D_{j}z)\right)\right)=K(|x|)f(z), \quad p \geq 2, && \text{ in } \Omega,\\
		&  z=\infty && \text{ on } \partial \Omega.
	\end{aligned}
	\right.
\end{equation}

Indeed, research on boundary blow-up solutions to $(N-1)$-Monge-Amp\`{e}re equations is still in its initial stage. Compared with the only existing work \cite{JiDengJiang2025}, the objectives, assumptions, and arguments in this work are entirely different. The method of sub- and super-solutions is a powerful analytical tool for establishing the existence of solutions to nonlinear problems, while also providing bounds by localizing the solution between the chosen sub- and super-solutions. The Keller-Osserman type condition $(\mathbf{f_{2}})$ plays a key role in our analysis and, as will be seen later, enables the construction of suitable sub- and super-solutions.

This paper is organized as follows. The necessary preliminaries are given in Section \ref{SectionPreliminaries}. Section \ref{SectionAnIVPandaComparisonPrinciple} presents an existence and uniqueness result for an IVP on $(0,1)$ and establishes a comparison principle for the corresponding differential inequalities. Finally, in Section \ref{SectionMultiplicityofRadialN1convexSolutions}, we prove the existence of infinitely many radial $(N-1)$-convex solutions to \eqref{Eqne1} by the method of sub- and super-solutions.

\section{Preliminaries}
\label{SectionPreliminaries}

The purpose of this section is to present the notations and preliminary facts. Throughout this paper, we denote $r=|x|=\sqrt{\sum_{i=1}^{N}x_{i}^{2}}$, where $x=(x_{1},x_{2},\ldots,x_{N}) \in \mathbb{R}^{N}$, and for $R>0$, define $B_{R}=\left\{x \in \mathbb{R}^{N} : |x|<R \right\}$.

For any closed bounded interval $[a,b] \subset \mathbb{R}$, the space $E=C[a,b]$ of all continuous functions $u : [a,b] \to \mathbb{R}$, with the supremum norm $\|u\|_{E}=\sup\{|u(\tau)| : \tau \in [a,b]\}$, is a Banach space.

The following lemma, together with its proof, is presented in \cite[Lemma $2.1$]{JiangJiLi2025} and plays a key role in transforming \eqref{Eqne1} into a form convenient for analysis.
\begin{lemma}\label{EVLemma001}
	Let $\zeta \in C^{2}[0,R)$ with $\zeta'(0)=0$. Then for $z(x) \coloneqq \zeta(r)$, $r=|x|$, we have $z \in C^{2}(B_{R})$, and the eigenvalues of $\Delta zI-D^{2}z$ are given by
	\begin{equation}
		\begin{split}
			\nonumber  &  \lambda \left(\Delta zI-D^{2}z\right)\\
			=&  \left \{
			\begin{aligned}
				& \left(\frac{N-1}{r}\zeta'(r),\zeta''(r)+\frac{N-2}{r}\zeta'(r),\ldots,\zeta''(r)+\frac{N-2}{r}\zeta'(r)\right), && r \in (0,R), \\
				& \left((N-1)\zeta''(0),(N-1)\zeta''(0),\ldots,(N-1)\zeta''(0)\right), && r=0.
			\end{aligned}
			\right. \nonumber
		\end{split}
	\end{equation}
	Furthermore, we have
	\begin{equation}
		\nonumber \det\left(\Delta zI-D^{2}z\right)=\left \{
		\begin{aligned}
			& \frac{N-1}{r}\zeta'(r)\left[\zeta''(r)+\frac{N-2}{r}\zeta'(r)\right]^{N-1}, && r \in (0,R), \\
			& \left((N-1)\zeta''(0)\right)^{N}, && r=0.
		\end{aligned}
		\right.
	\end{equation}
\end{lemma}
\begin{remark}
	Since $\zeta \in C^{2}[0,R)$ with $\zeta'(0)=0$, it follows that
	\begin{align}
		\nonumber  \lim_{r \to 0}\frac{\zeta'(r)}{r}=\lim_{r \to 0}\frac{\zeta'(r)-\zeta'(0)}{r-0}=\zeta''(0).
	\end{align}
	Therefore, we have $\left.\frac{\zeta'(r)}{r}\right|_{r=0}=\zeta''(0)$.
\end{remark}
We have the following useful result.
\begin{theorem}[The Keller-Osserman type condition]
	If $f : (0,\infty) \to (0,\infty)$ is continuous and nondecreasing on $(0,\infty)$, then the real $(N-1)$-Monge-Amp\`{e}re equation
	\begin{align}
		\nonumber  \operatorname{\det}^{\frac{1}{N-1}}\left(\Delta zI-D^{2}z\right)=f(z) \quad \text{ in } \mathbb{R}^{N}
	\end{align}
	admits a positive entire subsolution $z \in C^{2}(\mathbb{R}^{N})$ if and only if
	\begin{align}
		\nonumber  \int^{\infty}\left(\int_{0}^{\tau}f(s)ds\right)^{-\frac{N-1}{2N-1}}d\tau=\infty.
	\end{align}
\end{theorem}
\begin{proof}
	The proof is analogous to that of \cite[Theorem $1.1$]{JiangJiLi2025}.
\end{proof}
It should be noted that many of the integrals defined later are motivated by the Keller-Osserman type condition.

In the radially symmetric setting, the next result establishes an equivalence between radial solutions to \eqref{Eqne1} and solutions to a second-order ordinary differential equation with boundary conditions. This result is an immediate consequence of Lemma \ref{EVLemma001}.
\begin{lemma}\label{EVLemma002}
	Let $u \in C^{2}[0,1)$ with $u'(0)=0$. Define $z(x) \coloneqq u(r)$, where $r=|x|$. Then $z \in C^{2}(B_{1})$ is a radial solution to \eqref{Eqne1} if and only if $u$ is a solution to
	\begin{equation}\label{BBPEquation}
		\left \{
		\begin{aligned}
			& \left(\frac{N-1}{r}u'(r)\right)^{\frac{1}{N-1}}\left[u''(r)+\frac{N-2}{r}u'(r)\right]=K(r)f(u(r)), \quad r \in (0,1),\\
			& u'(0)=0, \quad u(1)=\infty.
		\end{aligned}
		\right.
	\end{equation}
\end{lemma}
In the following section, we analyze the resulting ordinary differential equation.

\section{An IVP and a Comparison Principle}
\label{SectionAnIVPandaComparisonPrinciple}

This section presents an existence and uniqueness result for an IVP and provides a comparison principle. We transform the IVP into an equivalent integral equation in a space of continuous functions, and to prove that the associated integral operator possesses a unique fixed point, we employ the Banach contraction principle\footnote[1]{Let $\mathfrak{B}$ be a nonempty closed subset of a Banach space $\mathfrak{X}$. If $\mathfrak{G} : \mathfrak{B} \to \mathfrak{B}$ is a contractive map, then $\mathfrak{G}$ has a unique fixed point in $\mathfrak{B}$.}.

Let $u_{0}>0$. Consider an IVP of the form
\begin{equation}\label{IVPEquation}
	\left \{
	\begin{aligned}
		& \left(\frac{N-1}{r}u'(r)\right)^{\frac{1}{N-1}}\left[u''(r)+\frac{N-2}{r}u'(r)\right]=K(r)f(u(r)), \quad r \in (0,1),\\
		& u(0)=u_{0}, \quad u'(0)=0.
	\end{aligned}
	\right.
\end{equation}
Our first result is as follows.
\begin{theorem}\label{EEPRSTheorem01}
	Suppose that conditions $(\mathbf{K})$ and $(\mathbf{f_{1}})$ hold. Then for every $u_{0}>0$, the IVP \eqref{IVPEquation} admits a unique solution $u \in C^{2}[0,T)$ on a maximal interval of existence $[0,T) \subset [0,1)$. Furthermore, $u'(r)>0$ for all $r \in (0,T)$, $u''(r)>0$ for all $r \in [0,T)$, and $u(r) \to \infty$ as $r \to T$ if $T<1$.
\end{theorem}
\begin{proof}
	We begin by noting that \eqref{IVPEquation} is equivalent to
	\begin{align}\label{IE01}
		u(r)=u_{0}+\int_{0}^{r}\frac{t^{2-N}}{N-1}\left(\int_{0}^{t}N\tau^{N-1}K(\tau)f(u(\tau))d\tau\right)^{\frac{N-1}{N}}dt.
	\end{align}
	In other words, every solution to the integral equation \eqref{IE01} is also a solution to the IVP \eqref{IVPEquation}, and vice versa.

	We first prove that \eqref{IVPEquation} admits a unique solution defined on $[0,h]$ for some sufficiently small $h>0$. We set $X=C[0,h]$ and define an operator $\mathcal{G} : X \to X$ by
	\begin{align}
		\nonumber \left(\mathcal{G}u\right)(r)=u_{0}+\int_{0}^{r}\frac{t^{2-N}}{N-1}\left(\int_{0}^{t}N\tau^{N-1}K(\tau)f(u(\tau))d\tau\right)^{\frac{N-1}{N}}dt
	\end{align}
	for all $r \in [0,h]$. We need to prove that the operator $\mathcal{G}$ has a unique fixed point in $X$. The proof is based on the Banach contraction principle. It suffices to show that $\mathcal{G}$ is contractive on a suitable subset of $X$.

	Denote
	\begin{align}
		\nonumber  K_{*}=\inf\left\{K(\tau) : \tau \in \left[0,\tfrac{1}{2}\right]\right\} \quad \text{and} \quad K^{*}=\sup\left\{K(\tau) : \tau \in \left[0,\tfrac{1}{2}\right]\right\}.
	\end{align}
	We introduce a set
	\begin{align}
		\nonumber  B_{h}(u_{0})=\left\{u \in X : \|u-u_{0}\|_{X} \leq h \right\}.
	\end{align}
	Choose $h^{*} \in \left(0,\tfrac{1}{2}\right)$ such that $u_{0}-h^{*}>0$. Since $f$ is a Lipschitz continuous on $[u_{0}-h^{*},u_{0}+h^{*}] \subset (0,\infty)$, there exists a constant $L>0$ such that
	\begin{align}
		\nonumber  |f(u_{1})-f(u_{2})| \leq L|u_{1}-u_{2}| \quad \text{ for every } u_{1},u_{2} \in [u_{0}-h^{*},u_{0}+h^{*}].
	\end{align}
	We use the monotonicity of $f$ and deduce 
	\begin{align}
		\nonumber  m \coloneqq f(u_{0}-h^{*}) \leq f(u) \leq f(u_{0}+h^{*}) \leq Lh^{*}+f(u_{0}) \eqqcolon M
	\end{align}
	for every $u \in [u_{0}-h^{*},u_{0}+h^{*}]$. It follows that there exists a sufficiently small constant $h^{\star} \in (0,h^{*})$ such that
	\begin{align}
		\nonumber  \frac{1}{2(N-1)}h^{2}\left(K^{*}M\right)^{\frac{N-1}{N}} \leq h \quad \text{ for all } h \in (0,h^{\star}].
	\end{align}
	We claim that for all $h \in (0,h^{\star}]$, $\mathcal{G}\left(B_{h}(u_{0})\right) \subset B_{h}(u_{0})$. In fact, for such an $h$ and every $u \in B_{h}(u_{0})$, it holds that
	\begin{align}
		\nonumber  |\left(\mathcal{G}u\right)(r)-u_{0}| &= \int_{0}^{r}\frac{t^{2-N}}{N-1}\left(\int_{0}^{t}N\tau^{N-1}K(\tau)f(u(\tau))d\tau\right)^{\frac{N-1}{N}}dt \\
		& \leq \int_{0}^{r}\frac{t^{2-N}}{N-1}\left(\int_{0}^{t}N\tau^{N-1}K^{*}Md\tau\right)^{\frac{N-1}{N}}dt \nonumber\\
		& = \frac{1}{N-1}\left(NK^{*}M\right)^{\frac{N-1}{N}}\int_{0}^{r}t^{2-N}\left(\int_{0}^{t}\tau^{N-1}d\tau\right)^{\frac{N-1}{N}}dt \nonumber\\
		& = \frac{1}{N-1}\left(K^{*}M\right)^{\frac{N-1}{N}}\frac{r^{2}}{2} \nonumber
	\end{align}
	for all $r \in [0,h]$. Taking the supremum over $r \in [0,h]$, this gives
	\begin{align}
		\nonumber  \|\mathcal{G}u-u_{0}\|_{X} \leq \frac{1}{2(N-1)}h^{2}\left(K^{*}M\right)^{\frac{N-1}{N}} \leq h,
	\end{align}
	which implies that $\mathcal{G}$ maps the set $B_{h}(u_{0})$ into itself.

	We now prove that the operator $\mathcal{G}$ is contractive on $B_{h}(u_{0})$ for all small $h>0$. For $h \in (0,h^{\star}]$ and every $u_{1}, u_{2} \in B_{h}(u_{0})$, we apply the mean value theorem and obtain 
	\begin{align}
		\nonumber  J(t) & \coloneqq \left(\int_{0}^{t}N\tau^{N-1}K(\tau)f(u_{1}(\tau))d\tau\right)^{\frac{N-1}{N}}-\left(\int_{0}^{t}N\tau^{N-1}K(\tau)f(u_{2}(\tau))d\tau\right)^{\frac{N-1}{N}}\\
		& = \frac{N-1}{N}\left(\int_{0}^{t}N\tau^{N-1}K(\tau)\left[\xi f(u_{1}(\tau))+(1-\xi)f(u_{2}(\tau))\right]d\tau\right)^{-\frac{1}{N}} \nonumber\\
		& ~~~ \times \int_{0}^{t}N\tau^{N-1}K(\tau)\left[f(u_{1}(\tau))-f(u_{2}(\tau))\right]d\tau \nonumber
	\end{align}
	with some $\xi=\xi(t) \in (0,1)$. Under the given conditions, for all $t \in [0,h]$,
	\begin{align}
		\nonumber  |J(t)| & \leq \frac{N-1}{N}\left(\int_{0}^{t}N\tau^{N-1}K_{*}md\tau\right)^{-\frac{1}{N}}\int_{0}^{t}N\tau^{N-1}K^{*}L|u_{1}(\tau)-u_{2}(\tau)|d\tau \\
		& \leq \frac{N-1}{N}\left(NK_{*}m\int_{0}^{t}\tau^{N-1}d\tau\right)^{-\frac{1}{N}}NK^{*}L\int_{0}^{t}\tau^{N-1}d\tau \|u_{1}-u_{2}\|_{X} \nonumber\\
		& = \frac{N-1}{N}\left(K_{*}m\right)^{-\frac{1}{N}}K^{*}Lt^{N-1}\|u_{1}-u_{2}\|_{X}. \nonumber
	\end{align}
	Now for all $r \in [0,h]$, it follows that
	\begin{align}
		\nonumber  |\left(\mathcal{G}u_{1}\right)(r)-\left(\mathcal{G}u_{2}\right)(r)| & = \left|\int_{0}^{r}\frac{t^{2-N}}{N-1}\left[\left(\int_{0}^{t}N\tau^{N-1}K(\tau)f(u_{1}(\tau))d\tau\right)^{\frac{N-1}{N}}\right.\right.\\
		& ~~~ -\left.\left.\left(\int_{0}^{t}N\tau^{N-1}K(\tau)f(u_{2}(\tau))d\tau\right)^{\frac{N-1}{N}}\right]dt\right| \nonumber\\
		& \leq \int_{0}^{r}\frac{t^{2-N}}{N-1}|J(t)|dt \nonumber\\
		& \leq \frac{1}{2N}\left(K_{*}m\right)^{-\frac{1}{N}}K^{*}Lr^{2}\|u_{1}-u_{2}\|_{X}. \nonumber
	\end{align}
	Taking the supremum over $r \in [0,h]$, this gives
	\begin{align}
		\nonumber  \|\mathcal{G}u_{1}-\mathcal{G}u_{2}\|_{X} \leq \frac{1}{2N}h^{2}\left(K_{*}m\right)^{-\frac{1}{N}}K^{*}L\|u_{1}-u_{2}\|_{X}.
	\end{align}
	Since $h \in (0,h^{\star}]$ is chosen sufficiently small, the inequality
	\begin{align}
		\nonumber  \frac{1}{2N}h^{2}\left(K_{*}m\right)^{-\frac{1}{N}}K^{*}L<1
	\end{align}
	holds, which implies that $\mathcal{G}$ is a contractive operator on $B_{h}(u_{0})$. By fixing such a small $h>0$, the Banach contraction principle asserts that $\mathcal{G}$ has a unique fixed point in $B_{h}(u_{0})$.

	We now need to prove that $u \in C^{2}[0,h]$. The regularity $u \in C^{2}(0,h]$ is immediate. Moreover, we have
	\begin{align}\label{FirstDerivative01}
		u'(r)=\frac{r^{2-N}}{N-1}\left(\int_{0}^{r}N\tau^{N-1}K(\tau)f(u(\tau))d\tau\right)^{\frac{N-1}{N}}
	\end{align}
	and
	\begin{align}\label{SecondDerivative02}
		\nonumber  u''(r) & = \frac{2-N}{N-1}r^{1-N}\left(\int_{0}^{r}N\tau^{N-1}K(\tau)f(u(\tau))d\tau\right)^{\frac{N-1}{N}}\\
		& ~~~ +\, rK(r)f(u(r))\left(\int_{0}^{r}N\tau^{N-1}K(\tau)f(u(\tau))d\tau\right)^{-\frac{1}{N}}
	\end{align}
	for all $r \in (0,h]$. Since $K>0$ on $[0,1)$ and $f>0$ on $(0,\infty)$, we have $u'(r), u''(r)>0$ for all $r \in (0,h]$. We proceed to establish the continuity of $u'(r)$ and $u''(r)$ at $r=0$. When $N=2$, by \eqref{IE01} and \eqref{FirstDerivative01}, we deduce	
	\begin{align}
		\nonumber  u'(0) &= \lim_{r \to 0}\frac{u(r)-u(0)}{r-0} \\
		&= \lim_{r \to 0}\frac{1}{r}\int_{0}^{r}\left(\int_{0}^{t}2\tau K(\tau)f(u(\tau))d\tau\right)^{\frac{1}{2}}dt \nonumber \\
		&= \lim_{r \to 0}\left(\int_{0}^{r}2\tau K(\tau)f(u(\tau))d\tau\right)^{\frac{1}{2}} \nonumber \\
		&= 0 \nonumber
	\end{align}
	and
	\begin{align}
		\nonumber  \lim_{r \to 0}u'(r) = \lim_{r \to 0}\left(\int_{0}^{r}2\tau K(\tau)f(u(\tau))d\tau\right)^{\frac{1}{2}} = 0.
	\end{align}
	When $N \geq 3$, it follows once again from \eqref{IE01} and \eqref{FirstDerivative01} that
	\begin{align}
		\nonumber  u'(0) & = \lim_{r \to 0}\frac{u(r)-u(0)}{r-0} \\
		& = \lim_{r \to 0}\frac{\int_{0}^{r}\frac{t^{2-N}}{N-1}\left(\int_{0}^{t}N\tau^{N-1}K(\tau)f(u(\tau))d\tau\right)^{\frac{N-1}{N}}dt}{r} \nonumber \\
		& = \frac{1}{N-1}\lim_{r \to 0}\frac{\left(\int_{0}^{r}N\tau^{N-1}K(\tau)f(u(\tau))d\tau\right)^{\frac{N-1}{N}}}{r^{N-2}} \nonumber \\
		& = \frac{1}{N-1}\left(\lim_{r \to 0}\frac{\int_{0}^{r}N\tau^{N-1}K(\tau)f(u(\tau))d\tau}{r^{\frac{N(N-2)}{N-1}}}\right)^{\frac{N-1}{N}} \nonumber \\
		& = \frac{1}{N-1}\left(\frac{N-1}{N-2}\lim_{r \to 0}r^{\frac{N}{N-1}}K(r)f(u(r))\right)^{\frac{N-1}{N}} \nonumber \\
		& = 0 \nonumber
	\end{align}
	and
	\begin{align}
		\nonumber  \lim_{r \to 0}u'(r) = \lim_{r \to 0}\frac{r^{2-N}}{N-1}\left(\int_{0}^{r}N\tau^{N-1}K(\tau)f(u(\tau))d\tau\right)^{\frac{N-1}{N}}=0.
	\end{align}
	In consequence, we have $\lim_{r \to 0}u'(r)=u'(0)=0$, and thus, $u \in C^{1}[0,h]$. Applying \eqref{FirstDerivative01} again, we obtain
	\begin{align}
		\nonumber  u''(0) & = \lim_{r \to 0}\frac{u'(r)-u'(0)}{r-0} \\
		& = \lim_{r \to 0}\frac{\frac{r^{2-N}}{N-1}\left(\int_{0}^{r}N\tau^{N-1}K(\tau)f(u(\tau))d\tau\right)^{\frac{N-1}{N}}}{r} \nonumber \\
		& = \frac{1}{N-1}\left(\lim_{r \to 0}\frac{\int_{0}^{r}N\tau^{N-1}K(\tau)f(u(\tau))d\tau}{r^{N}}\right)^{\frac{N-1}{N}} \nonumber \\
		& = \frac{1}{N-1}\left(\lim_{r \to 0}K(r)f(u(r))\right)^{\frac{N-1}{N}} \nonumber \\
		& = \frac{1}{N-1}\left(K(0)f(u(0))\right)^{\frac{N-1}{N}}. \nonumber
	\end{align}
	Clearly, $u''(0)>0$. A direct calculation using \eqref{SecondDerivative02} yields
	\begin{align}
		\nonumber  \lim_{r \to 0}u''(r) & = \frac{2-N}{N-1}\left(\lim_{r \to 0}\frac{\int_{0}^{r}N\tau^{N-1}K(\tau)f(u(\tau))d\tau}{r^{N}}\right)^{\frac{N-1}{N}} \nonumber \\
		& ~~~ +\, \lim_{r \to 0}\left[K(r)f(u(r))\right]\left(\lim_{r \to 0}\frac{\int_{0}^{r}N\tau^{N-1}K(\tau)f(u(\tau))d\tau}{r^{N}}\right)^{-\frac{1}{N}} \nonumber\\
		& = \frac{2-N}{N-1}\left(K(0)f(u(0))\right)^{\frac{N-1}{N}}+\left(K(0)f(u(0))\right)^{\frac{N-1}{N}} \nonumber\\
		& = \frac{1}{N-1}\left(K(0)f(u(0))\right)^{\frac{N-1}{N}}. \nonumber
	\end{align}
	Consequently, we have $\lim_{r \to 0}u''(r)=u''(0)$, and thus, $u \in C^{2}[0,h]$. Hence, $u$ is the unique solution to the IVP \eqref{IVPEquation} on the interval $[0,h]$.

	In order to extend the solution $u(r)$ to $r>h$, we let $u'=v$ and define $W=\begin{pmatrix}
		v\\
		u
	\end{pmatrix}$. It follows from \eqref{IE01} and \eqref{SecondDerivative02} that
	\begin{align}
		\nonumber  v'(r)=u''(r)=\frac{2-N}{r}v(r)+rK(r)f(u(r))\left(\frac{(N-1)v(r)}{r^{2-N}}\right)^{-\frac{1}{N-1}}.
	\end{align}
	Hence, we obtain the system of first-order differential equations
	\begin{equation}\label{SystemofEqus01}
		\left \{
		\begin{aligned}
			& W'(r)=\begin{pmatrix}
				K(r)f(u(r))\left(\frac{r}{(N-1)v(r)}\right)^{\frac{1}{N-1}}-\frac{N-2}{r}v(r)\\
				v(r)
			\end{pmatrix} \eqqcolon \mathscr{F}(r,W(r)),\\
			& W(h)=\begin{pmatrix}
				u'(h)\\
				u(h)
			\end{pmatrix}.
		\end{aligned}
		\right.
	\end{equation}
	Under assumptions $(\mathbf{K})$ and $(\mathbf{f_{1}})$, $\mathscr{F}$ is continuous in $r$ on $[0,1)$, and $\mathscr{F}(r,W)$ is locally Lipschitz continuous in $W$ for $W > 0$. We therefore infer that system \eqref{SystemofEqus01} admits a unique solution defined on a small neighborhood of $h$. Notice that the component $u$ of $W$ satisfies
	\begin{align}
		\nonumber  \left(\frac{N-1}{r}u'(r)\right)^{\frac{1}{N-1}}\left[u''(r)+\frac{N-2}{r}u'(r)\right]=K(r)f(u(r))>0,
	\end{align}
	with $u(h)>0$ and $u'(h)>0$. Thus, we have $u'(r)>u'(h)$ and $u''(r)>0$ for $r>h$. Therefore, the solution $W$ to system \eqref{SystemofEqus01} can be extended to $r>h$ until either $r$ reaches $1$ or $u(r)$ blows up to infinity. Hence, the IVP \eqref{IVPEquation} admits a unique solution $u$ on some maximal interval of existence $[0,T)$ with $T \leq 1$, and $u(r) \to \infty$ as $r \to T$ if $T<1$. This completes the proof.
\end{proof}

The next result is a comparison principle that will play an important role in the sequel.
\begin{theorem}[Comparison principle]\label{EPEBRSTheorem01}
	Suppose that conditions $(\mathbf{K})$ and $(\mathbf{f_{1}})$ hold. Let $u_{1}, u_{2} \in C^{2}[0,T)$ with $u_{1}'(0)=u_{2}'(0)=0$, and assume that for all $r \in (0,T)$ the following inequalities are satisfied:
	\begin{align}\label{Inequality001}
		\left(\frac{N-1}{r}u_{1}'(r)\right)^{\frac{1}{N-1}}\left[u_{1}''(r)+\frac{N-2}{r}u_{1}'(r)\right] \leq K(r)f(u_{1}(r)) 
	\end{align}
	and
	\begin{align}\label{Inequality002}
		\left(\frac{N-1}{r}u_{2}'(r)\right)^{\frac{1}{N-1}}\left[u_{2}''(r)+\frac{N-2}{r}u_{2}'(r)\right] \geq K(r)f(u_{2}(r)).
	\end{align}
	Then $u_{1}(0)<u_{2}(0)$ implies $u_{1}(r)<u_{2}(r)$ for all $r \in [0,T)$. 
\end{theorem}
\begin{proof}
	The proof proceeds by contradiction. Assume that $u_{1}<u_{2}$ on $[0,T)$ does not hold. Since $u_{1}, u_{2} \in C^{2}[0,T)$ and $u_{1}(0)<u_{2}(0)$, there exists a point $\widehat{T} \in (0,T)$ such that $u_{1}(r)<u_{2}(r)$ for all $r \in [0,\widehat{T})$ and $u_{1}(\widehat{T})=u_{2}(\widehat{T})$. On the other hand, in view of \eqref{Inequality001} and \eqref{Inequality002}, we have
	\begin{align}
		\nonumber  u_{1}(r) \leq u_{1}(0)+\int_{0}^{r}\frac{t^{2-N}}{N-1}\left(\int_{0}^{t}N\tau^{N-1}K(\tau)f(u_{1}(\tau))d\tau\right)^{\frac{N-1}{N}}dt
	\end{align}
	and
	\begin{align}
		\nonumber  u_{2}(r) \geq u_{2}(0)+\int_{0}^{r}\frac{t^{2-N}}{N-1}\left(\int_{0}^{t}N\tau^{N-1}K(\tau)f(u_{2}(\tau))d\tau\right)^{\frac{N-1}{N}}dt,
	\end{align}
	respectively. From the monotonicity of $f$ and the condition $u_{1}(0)<u_{2}(0)$, it follows that
	\begin{align}
		\nonumber  u_{1}(\widehat{T}) & \leq u_{1}(0)+\int_{0}^{\widehat{T}}\frac{t^{2-N}}{N-1}\left(\int_{0}^{t}N\tau^{N-1}K(\tau)f(u_{1}(\tau))d\tau\right)^{\frac{N-1}{N}}dt\\
		& < u_{2}(0)+\int_{0}^{\widehat{T}}\frac{t^{2-N}}{N-1}\left(\int_{0}^{t}N\tau^{N-1}K(\tau)f(u_{2}(\tau))d\tau\right)^{\frac{N-1}{N}}dt \nonumber\\
		& \leq u_{2}(\widehat{T}). \nonumber
	\end{align}
	However, this contradicts the fact that $u_{1}(\widehat{T})=u_{2}(\widehat{T})$. The proof is complete.
\end{proof}

\section{Multiplicity of Radial $(N-1)$-convex Solutions}
\label{SectionMultiplicityofRadialN1convexSolutions}

In this section, we establish the existence of infinitely many radial $(N-1)$-convex solutions to the boundary blow-up problem \eqref{Eqne1}. The proof relies on the method of sub- and super-solutions, in which the Keller-Osserman type condition $(\mathbf{f_{2}})$ is a central ingredient. We begin by introducing the notations and their basic facts required for the subsequent discussion.

As a consequence of \eqref{SecEqu01}, there exists $a>0$ such that
\begin{align}\label{MRExEquation01}
	G(t) \coloneqq \int_{a}^{t}\left(\frac{2N-1}{N-1}F(\tau)\right)^{-\frac{N-1}{2N-1}}d\tau \to \infty \text{ as } t \to \infty.
\end{align}
Denote by $g$ the inverse function of $G$; that is, $g$ satisfies
\begin{align}\label{MRExEquation02}
	\int_{a}^{g(t)}\left(\frac{2N-1}{N-1}F(\tau)\right)^{-\frac{N-1}{2N-1}}d\tau=t \quad \text{ for all } t>0.
\end{align}
It is easy to see that
\begin{align}
	\nonumber  g(0)=\lim_{t \to 0^{+}}g(t)=a \quad \text{and} \quad \lim_{t \to \infty}g(t)=\infty.
\end{align}
By differentiating \eqref{MRExEquation02} with respect to $t$, we obtain
\begin{align}
	\nonumber g'(t)=\left(\frac{2N-1}{N-1}F(g(t))\right)^{\frac{N-1}{2N-1}}.
\end{align}
Then
\begin{align}
	\nonumber  g''(t)=\frac{f(g(t))}{\left(\frac{2N-1}{N-1}F(g(t))\right)^{\frac{1}{2N-1}}}.
\end{align}
Thus, there hold
\begin{align}
	\nonumber  \left(g'(t)\right)^{\frac{1}{N-1}}g''(t)=f(g(t))
\end{align}
and
\begin{align}\label{MRExEquation03}
	\frac{g'(t)}{g''(t)}=\frac{\left(\frac{2N-1}{N-1}F(g(t))\right)^{\frac{N}{2N-1}}}{f(g(t))}=-\frac{\left[G'(g(t))\right]^{2}}{G''(g(t))}.
\end{align}
Define
\begin{align}\label{MRExEquation04}
	H(\tau)=-\frac{G(\tau)G''(\tau)}{\left[G'(\tau)\right]^{2}}.
\end{align}
In what follows, let $p \in C^{1}(0,\infty)$ be positive, satisfying $p'(t)<0$ for all $t>0$ and $\lim_{t \to 0^{+}}p(t)=\infty$. Then we define $P$ by
\begin{align}
	\nonumber  P(\tau)=\int_{\tau}^{1}p(s)ds.
\end{align}
We say that the function $p$ is of class $\mathcal{P}_{\infty}$ if
\begin{align}\label{MRExEquation05}
	\int_{0^{+}}\left(P(\tau)\right)^{\frac{N-1}{N}}d\tau=\infty
\end{align}
(\footnote[2]{$\int_{0^{+}}h(\tau)d\tau \coloneqq \lim_{\delta \to 0^{+}}\int_{\delta}^{b}h(\tau)d\tau$, for some fixed $b>0$.}).

We are now in a position to prove the existence of infinitely many boundary blow-up solutions to \eqref{Eqne1}. We construct sub- and super-solutions on the given interval, and the comparison principle ensures control of the solutions between them.
\begin{theorem}\label{EPEBURSTheorem}
	Suppose that conditions $(\mathbf{K})$, $(\mathbf{f_{1}})$, and $(\mathbf{f_{2}})$ hold. Assume further that there exist a function $p \in \mathcal{P}_{\infty}$ and two constants $c,d>0$ such that
	\begin{align}
		\nonumber  cp(1-r) \leq K(r) \leq dp(1-r) \quad \text{ for all } r<1 \text{ close to } 1.
	\end{align}
	If $\lim_{\tau \to \infty}H(\tau)$ exists (denoted by $H_{\infty}$), then \eqref{Eqne1} has infinitely many radial $(N-1)$-convex solutions.
\end{theorem}
\begin{proof}
	The condition on $p$ ensures that \eqref{MRExEquation05} holds. Define
	\begin{align}\label{MREquation01}
		\varphi(t) \coloneqq \int_{t}^{1}\left(\frac{N}{N-1}P(\tau)\right)^{\frac{N-1}{N}}d\tau \quad \text{ for } t>0.
	\end{align}
	Note from \eqref{MRExEquation05} that
	\begin{align}
		\nonumber  \varphi(0)=\lim_{t \to 0^{+}}\varphi(t)=\infty.
	\end{align}
	Moreover, we have
	\begin{align}\label{MREquation02}
		\varphi'(t)=-\left(\frac{N}{N-1}P(t)\right)^{\frac{N-1}{N}} \quad \text{and} \quad \varphi''(t)=p(t)\left(\frac{N}{N-1}P(t)\right)^{-\frac{1}{N}}.
	\end{align}
	It is easy to see that the function $y(r)=\frac{1-r^{2}}{2}$ satisfies
	\begin{equation}
		\nonumber  \left \{
		\begin{aligned}
			& (-1)^{N}y'(r)\left[y''(r)\right]^{N-1}=r, \quad r \in (0,1),\\
			& y'(0)=0, \quad y(1)=0.
		\end{aligned}
		\right.
	\end{equation}
	For $r \in [0,1)$, define
	\begin{align}
		\nonumber  w(r) \coloneqq g(k\varphi^{\frac{N}{2N-1}}(y(r))) \quad \text{ for some constant } k>0.
	\end{align}
	An elementary calculation, using the chain rule, yields that
	\begin{align}
		\nonumber  w'=\frac{kN}{2N-1}g'\varphi^{-\frac{N-1}{2N-1}}\varphi' y'
	\end{align}
	and
	\begin{align}
		\nonumber  w'' & = \frac{kN}{2N-1}\varphi^{-\frac{N-1}{2N-1}}\left[\frac{kN}{2N-1}g''\varphi^{-\frac{N-1}{2N-1}}(\varphi')^{2}(y')^{2}\right.\\
		& ~~~ -\left.\frac{N-1}{2N-1}g'\frac{(\varphi')^{2}}{\varphi}(y')^{2}+g'\varphi''(y')^{2}+g'\varphi' y''\right]. \nonumber
	\end{align}
	Consequently, we have
	\begin{align}\label{MREquation03}
		\nonumber  (w')^{\frac{1}{N-1}}w'' & = k^{\frac{2N-1}{N-1}}\left(\frac{N}{2N-1}\right)^{\frac{N}{N-1}}(g')^{\frac{1}{N-1}}g''(-\varphi')^{\frac{1}{N-1}}\varphi''(-1)^{\frac{N}{N-1}}(y')^{\frac{1}{N-1}}y''\\
		& ~~~ \times\left[\frac{N}{2N-1}\frac{(\varphi')^{2}}{\varphi\varphi''}\left(-\frac{(y')^{2}}{y''}\right)-\frac{N-1}{2N-1}\frac{g'}{k\varphi^{\frac{N-1}{2N-1}}g''}\frac{(\varphi')^{2}}{\varphi\varphi''}\left(-\frac{(y')^{2}}{y''}\right)\right. \nonumber\\
		& ~~~ +\left.\frac{g'}{k\varphi^{\frac{N-1}{2N-1}}g''}\left(-\frac{(y')^{2}}{y''}\right)+\frac{g'}{k\varphi^{\frac{N-1}{2N-1}}g''}\left(-\frac{\varphi'}{\varphi''}\right)\right].
	\end{align}
	The definition of $w$ gives
	\begin{align}\label{MREquation04}
		k\varphi^{\frac{N-1}{2N-1}}(y(r))=g^{-1}(w(r))=G(w(r)).
	\end{align}
	From \eqref{MRExEquation03}, \eqref{MRExEquation04}, and \eqref{MREquation04}, we obtain that
	\begin{align}\label{MREquation05}
		\frac{1}{H(w)}=\frac{g'}{k\varphi^{\frac{N-1}{2N-1}}g''}.
	\end{align}
	It follows from \eqref{MREquation02} that
	\begin{align}\label{MREquation06}
		(-\varphi'(t))^{\frac{1}{N-1}}\varphi''(t)=p(t) \quad \text{and} \quad \frac{\varphi'(t)}{\varphi''(t)}=-\frac{N}{N-1}\frac{P(t)}{p(t)}.
	\end{align}
	We also have
	\begin{align}\label{MREquation07}
		y'(r)=-r \quad \text{and} \quad y''(r)=-1.
	\end{align}
	Now using \eqref{MREquation05}--\eqref{MREquation07}, equation \eqref{MREquation03} reduces to
	\begin{align}
		\nonumber  (w')^{\frac{1}{N-1}}w''=k^{\frac{2N-1}{N-1}}\left(\frac{N}{2N-1}\right)^{\frac{N}{N-1}}r^{\frac{1}{N-1}}p(y)f(w)\Delta(r),
	\end{align}
	where
	\begin{align}
		\nonumber  \Delta(r) & \coloneqq \frac{N}{2N-1}\frac{1}{S(y)}r^{2}-\frac{N-1}{2N-1}\frac{1}{H(w)}\frac{1}{S(y)}r^{2}+\frac{1}{H(w)}r^{2}\\
		& ~~~ + \, \frac{N}{N-1}\frac{1}{H(w)}\frac{P(y)}{p(y)}, \nonumber
	\end{align}
	with
	\begin{align}\label{MREquation08}
		S(\tau)=\frac{\varphi(\tau)\varphi''(\tau)}{\left[\varphi'(\tau)\right]^{2}}.
	\end{align}
	Clearly, there holds
	\begin{align}
		\nonumber  (w')^{\frac{N}{N-1}}=k^{\frac{2N-1}{N-1}}\left(\frac{N}{2N-1}\right)^{\frac{N}{N-1}}\frac{N}{N-1}r^{\frac{N}{N-1}}\frac{1}{H(w)}P(y)f(w).
	\end{align}
	Therefore, we deduce
	\begin{align}\label{MREquation09}
		\nonumber  \left(\frac{N-1}{r}w'\right)^{\frac{1}{N-1}}\left[w''+\frac{N-2}{r}w'\right] & = (N-1)^{\frac{1}{N-1}}k^{\frac{2N-1}{N-1}}\left(\frac{N}{2N-1}\right)^{\frac{N}{N-1}}p(y)f(w)\\
		& ~~~ \times \left[\Delta(r)+\frac{N(N-2)}{N-1}\frac{1}{H(w)}\frac{P(y)}{p(y)}\right].
	\end{align}
	From \eqref{MREquation01} and \eqref{MREquation02}, we derive
	\begin{align}
		\nonumber  \frac{\left[\varphi'(t)\right]^{2}}{\varphi(t)\varphi''(t)} & = \frac{\left(\frac{N}{N-1}P(t)\right)^{\frac{2N-1}{N}}}{p(t)\int_{t}^{1}\left(\frac{N}{N-1}P(\tau)\right)^{\frac{N-1}{N}}d\tau}\\
		& = \frac{\int_{t}^{1}\left[\left(\frac{N}{N-1}P(s)\right)^{\frac{2N-1}{N}}\right]'ds}{\int_{t}^{1}\left[p(s)\int_{s}^{1}\left(\frac{N}{N-1}P(\tau)\right)^{\frac{N-1}{N}}d\tau\right]'ds} \nonumber\\
		& = \frac{\frac{2N-1}{N-1}\int_{t}^{1}p(s)\left(\frac{N}{N-1}P(s)\right)^{\frac{N-1}{N}}ds}{\int_{t}^{1}\left[p(s)\left(\frac{N}{N-1}P(s)\right)^{\frac{N-1}{N}}-p'(s)\int_{s}^{1}\left(\frac{N}{N-1}P(\tau)\right)^{\frac{N-1}{N}}d\tau\right]ds} \nonumber\\
		& \leq \frac{2N-1}{N-1}, \nonumber
	\end{align}
	which, in turn, implies
	\begin{align}
		\nonumber  \frac{1}{H(w)}-\frac{N-1}{2N-1}\frac{1}{H(w)}\frac{1}{S(y)} \geq 0.
	\end{align}
	Therefore, we infer that
	\begin{align}
		\nonumber  \Theta(r) \coloneqq \frac{N}{2N-1}\frac{1}{S(y)}r^{2}-\frac{N-1}{2N-1}\frac{1}{H(w)}\frac{1}{S(y)}r^{2}+\frac{1}{H(w)}r^{2} \geq 0,
	\end{align}
	and
	\begin{align}
		\nonumber  \Theta(r)>0 \quad \text{ for all } r \in (0,1].
	\end{align}
	Since
	\begin{align}
		\nonumber  \lim_{t \to 0}\frac{P(t)}{p(t)}=0, \text{ which implies } \lim_{r \to 1}\frac{P(y(r))}{p(y(r))}=0, \text{ and } H_{\infty} \neq \infty,
	\end{align}
	it follows that $\Delta(r)$ is positive and continuous for all $r \in [0,1)$ and admits a finite positive limit as $r \to 1^{-}$. Then there exist constants $C_{1}$ and $C_{2}$ with $0<C_{1}<C_{2}$ such that
	\begin{align}\label{MREquation0010}
		C_{1} \leq \Delta(r)+\frac{N(N-2)}{N-1}\frac{1}{H(w)}\frac{P(y)}{p(y)} \leq C_{2} \quad \text{ for all } r \in [0,1).
	\end{align}
	Combining \eqref{MREquation09} with inequality \eqref{MREquation0010}, we deduce that
	\begin{align}\label{MREquation0011}
		\left(\frac{N-1}{r}w'\right)^{\frac{1}{N-1}}\left[w''+\frac{N-2}{r}w'\right] \leq (N-1)^{\frac{1}{N-1}}k^{\frac{2N-1}{N-1}}\left(\frac{N}{2N-1}\right)^{\frac{N}{N-1}}C_{2}p(y)f(w) 
	\end{align}
	and
	\begin{align}\label{MREquation0012}
		\left(\frac{N-1}{r}w'\right)^{\frac{1}{N-1}}\left[w''+\frac{N-2}{r}w'\right] \geq (N-1)^{\frac{1}{N-1}}k^{\frac{2N-1}{N-1}}\left(\frac{N}{2N-1}\right)^{\frac{N}{N-1}}C_{1}p(y)f(w)
	\end{align}
	for all $r \in [0,1)$. By replacing $p(t)$ with $\epsilon p(2t)$ for some sufficiently small $\epsilon>0$, we may assume that
	\begin{align}
		\nonumber  K(r) \geq p\left(\frac{1-r}{2}\right) \quad \text{ for all } r \in [0,1).
	\end{align}
	Since $y(r) \geq \frac{1-r}{2}$, it follows that
	\begin{align}\label{MREquation0013}
		p(y(r)) \leq p\left(\frac{1-r}{2}\right) \leq K(r) \quad \text{ for all } r \in [0,1).
	\end{align}
	Inserting \eqref{MREquation0013} into \eqref{MREquation0011}, we obtain
	\begin{align}
		\nonumber  \left(\frac{N-1}{r}w'\right)^{\frac{1}{N-1}}\left[w''+\frac{N-2}{r}w'\right] \leq (N-1)^{\frac{1}{N-1}}k^{\frac{2N-1}{N-1}}\left(\frac{N}{2N-1}\right)^{\frac{N}{N-1}}C_{2}K(r)f(w)
	\end{align}
	for all $r \in [0,1)$. Define
	\begin{align}
		\nonumber  w_{1}(r) \coloneqq g(k_{1}\varphi^{\frac{N}{2N-1}}(y(r))),
	\end{align}
	where $k_{1}>0$ is a constant. We can choose $k_{1}$ sufficiently small so that $w_{1}$ satisfies
	\begin{align}
		\nonumber  \left(\frac{N-1}{r}w_{1}'(r)\right)^{\frac{1}{N-1}}\left[w_{1}''(r)+\frac{N-2}{r}w_{1}'(r)\right] \leq K(r)f(w_{1}(r))
	\end{align}
	for all $r \in [0,1)$.

	We now proceed to construct a function $w_{2}$ that satisfies the reverse inequality. By replacing $p(t)$ with $Mp(t)$ for some sufficiently large $M>0$, we may assume that
	\begin{align}
		\nonumber  p(1-r) \geq K(r) \quad \text{ for all } r \in [0,1).
	\end{align}
	Since $y(r) \leq 1-r$, it follows that
	\begin{align}
		\nonumber  p(y(r)) \geq p(1-r) \geq K(r) \quad \text{ for all } r \in [0,1).
	\end{align}
	Thus, by \eqref{MREquation0012} (with $\varphi(t)$ and $C_{1}$ determined by this new function $p(t)$), we have
	\begin{align}
		\nonumber  \left(\frac{N-1}{r}w'\right)^{\frac{1}{N-1}}\left[w''+\frac{N-2}{r}w'\right] \geq (N-1)^{\frac{1}{N-1}}k^{\frac{2N-1}{N-1}}\left(\frac{N}{2N-1}\right)^{\frac{N}{N-1}}C_{1}K(r)f(w)
	\end{align}
	for all $r \in [0,1)$. Furthermore, by choosing $k=k_{2}$ sufficiently large, we define
	\begin{align}
		\nonumber  w_{2}(r) \coloneqq g(k_{2}\varphi^{\frac{N}{2N-1}}(y(r))) \quad \text{ for all } r \in [0,1),
	\end{align}
	which satisfies
	\begin{align}
		\nonumber  w_{1}(0)<w_{2}(0), \quad \left(\frac{N-1}{r}w_{2}'(r)\right)^{\frac{1}{N-1}}\left[w_{2}''(r)+\frac{N-2}{r}w_{2}'(r)\right] \geq K(r)f(w_{2}(r))
	\end{align}
	for all $r \in [0,1)$. For any $k \in (w_{1}(0),w_{2}(0))$, let $u_{k}$ denote the unique solution to the IVP \eqref{IVPEquation} with $u_{0}=k$. Then in view of Theorem \ref{EPEBRSTheorem01}, we have
	\begin{align}
		\nonumber  w_{1}(r)<u_{k}(r)<w_{2}(r) \quad \text{ for all } r \in [0,1)
	\end{align}
	and $u_{k}(r)$ is well-defined. We therefore conclude, by Theorem \ref{EEPRSTheorem01}, that $u_{k}(r)$ is defined for all $r \in [0,1)$, with $u_{k}'(r)>0$ for all $r \in (0,1)$ and $u_{k}''(r)>0$ for all $r \in [0,1)$. Since $w_{1}(r) \to \infty$ as $r \to 1$, it follows that $u_{k}(r) \to \infty$ as $r \to 1$. This shows that $u_{k}$ is indeed a solution to \eqref{BBPEquation}. Varying $k$ thus yields infinitely many solutions to \eqref{BBPEquation}; equivalently, by Lemma \ref{EVLemma002}, problem \eqref{Eqne1} admits infinitely many radial $(N-1)$-convex solutions. This completes the proof.  
\end{proof}

\section*{Declaration of competing interest}
There is no conflict of interest.

\section*{Acknowledgements}
The first author gratefully acknowledges the financial support provided by the Research and Development Cell, Indian Institute of Technology Guwahati, India, through Project No. MATHSPNIITG01179xxKS001.

\section*{Data availability}
No data was used for the research described in the article.

\bibliographystyle{plain}

\end{document}